\documentclass[12pt,a4paper]{article}
\usepackage{mathrsfs}
\usepackage{epsfig, graphicx}
\usepackage{latexsym,amsfonts,amsbsy,amssymb}
\usepackage{amsmath,amsthm}
\usepackage{color}
\usepackage[colorlinks, citecolor=blue]{hyperref}
\makeatletter 

\@addtoreset{equation}{section}
\makeatother 
\allowdisplaybreaks 

\def\cA{\mathcal{A}}

\def\cE{\mathcal{E}}

\def\cJ{\mathcal{J}}

\def\cT{\mathcal{T}}
\def\cR{\mathcal{R}}

\def\w\eta{\widetilde{\eta}}

\newcommand{\bc}{\begin{center}}
\newcommand{\ec}{\end{center}}
\newcommand{\be}{\begin{eqnarray}}
\newcommand{\ee}{\end{eqnarray}}

\newcommand{\ben}{\begin{eqnarray*}}
\newcommand{\een}{\end{eqnarray*}}

\newtheorem{theorem}{Theorem}[section]
\newtheorem{lemma}{Lemma}[section]
\newtheorem{corollary}{Corollary}[section]
\newtheorem{remark}{Remark}[section]
\newtheorem{algorithm}{Algorithm}[section]

\begin{document}
\title{A Full Multigrid Method For Semilinear Elliptic Equation\footnote{This work
was supported in part by National Natural Science Foundations of China
(NSFC 91330202, 11371026, 11001259, 11031006, 2011CB309703)
and the National Center for Mathematics and Interdisciplinary Science, CAS.}}
\author{
Hehu Xie\footnote{LSEC, ICMSEC,
Academy of Mathematics and Systems Science, Chinese Academy of
Sciences, Beijing 100190, China (hhxie@lsec.cc.ac.cn)}
\ \ and \
Fei Xu\footnote{Beijing Institute for Scientific and Engineering
        Computing, Beijing University of Technology,
        Beijing 100124, China
        (xufei@lsec.cc.ac.cn)}
}
\date{}
\maketitle
\begin{abstract}
A full multigrid  finite element method is proposed for semilinear elliptic equations.
The main idea is to transform the solution of the semilinear problem into a series of
solutions of the corresponding linear boundary value problems on the sequence of finite
element spaces and semilinear problems on a very low dimensional space. The linearized
boundary value problems are solved by some multigrid iterations.  Besides the multigrid
iteration, all other efficient numerical methods can also serve as the linear solver for solving
boundary value problems. The optimality of the computational work is also proved.
Compared with the existing multigrid methods which need the
bounded second order derivatives of the nonlinear term, the proposed method only needs the
Lipschitz continuation in some sense of the nonlinear term.


\end{abstract}
\section{Introduction}
The purpose of this paper is to study the multigird finite element method for semilinear elliptic problems.
As we know, the multigrid and multilevel methods \cite{Bramble,BramblePasciak,BrambleZhang,BrandtMcCormickRuge,
Hackbusch_Book,ScottZhang,ToselliWidlund,shaidurov1995multigrid,Xu} provide optimal order algorithms for solving
boundary value problems. The error bounds of the approximate solutions obtained from these efficient numerical
algorithms are comparable to the theoretical bounds determined by the finite element discretization.
In the past decade years, some researches about multigrid method for nonlinear
elliptic problem are studied to improve the efficiency of nonlinear elliptic problem solving,
i.e. \cite{shaidurov1995multigrid,J Xu1,J Xu2}.
The Newton iteration is adopted to linearize the nonlinear equation in these existing multigrid methods
and then they need the bounded second order derivatives of the nonlinear terms.
For more information, please refer to \cite{HuangShiTangXue,shaidurov1995multigrid,J Xu1} and the references cited therein.

Recently, a type of multigrid method with optimal efficiency for eigenvalue problems has been proposed
in \cite{LinXie,Xie_IMA,Xie_JCP,XieXie}. The aim of this paper is to present a full multigrid method 
for solving semilinear elliptic problems based on the multilevel correction
scheme \cite{Xie_IMA,Xie_JCP}. The main idea is to design a special low dimensional space to 
transform the solution of the semilinear problem into a series of solutions of the corresponding linear boundary
value problems on the sequence of finite element spaces and semilinear  problems on a very low dimensional space.
For the linearized elliptic problem,
it is not necessary to solve the linear boundary value problem
exactly in each correction step.
Here, we only do some multigrid iteration steps for the linear boundary value problems.
In this new version of multigrid method, solving semilinear elliptic problem will not be much more difficult
than the multigrid scheme for the corresponding linear boundary value problems.
Compared with the existing multigrid methods for the semilinear problem, our method only
needs the Lipschitz continuation in some sense of the nonlinear term. 

An outline of the paper goes as follows. In Section 2, we introduce the
finite element method for the semilinear elliptic problem.
A type of full multigrid method  for the semilinear
elliptic problem is given in Section 3. In Section 4, some numerical examples are provided to
validate the efficiency of the proposed numerical method.
Some concluding remarks are given in the last section.
\section{Discretization by finite element method}
In this paper, the letter $C$ (with or without subscripts) is used to
denote a constant which may be different at
different places. For convenience, the symbols $ x_1 \lesssim y_1$,
$x_2 \gtrsim y_2$ and $x_3 \approx y_3$ mean that $x_1 \leq C_1y_1$, $x_2 \geq c_2y_2$ and
$c_3x_3 \leq y_3 \leq C_3 x_3$. Let $\Omega \subset \cR^d\ (d = 2, 3 )$
denote a bounded convex domain with Lipschitz boundary $\partial\Omega$. We use the standard notation
for Sobolev spaces $W^{s,p}(\Omega)$
and their associated norms $\|\cdot\|_{s,p,\Omega}$ and seminorms
 $|\cdot|_{s,p,\Omega}$ (see, e.g. \cite{Adams}).
For $p = 2$, we denote $H^s(\Omega) = W^{s,2}(\Omega)$ and
$H_0^1(\Omega) = \{v \in H^1(\Omega):v|_{\partial \Omega} = 0\}$,
where $v|_{\partial \Omega} = 0$ is in the sense of trace.
For simplicity, we use $\|\cdot\|_{s}$ to
denote $\|\cdot\|_{s,2,\Omega}$ and $V$ to denote $H_0^1(\Omega)$ in the rest of the paper.

Here, we consider the following type of semilinear elliptic equation:
\begin{equation}\label{semilinear equation}
\left\{
\begin{array}{rcl}
-\nabla\cdot(\mathcal{A}\nabla u)+f(x,u)&=& g, \quad \text{in } \Omega,\\
u&=& 0,\quad \text{on } {\partial\Omega},
\end{array}
\right.
\end{equation}
where $\cA=(a_{i,j})_{d\times d}$ is a symmetric positive definite
matrix with $a_{i,j}\in W^{1,\infty} \ \ (i,j=1,2,\cdots ,d)$,
$f(x,u)$ is a nonlinear function with respect to the second variable.

The weak form of the semilinear problem (\ref{semilinear equation}) can
be described as: Find $u\in V$ such that
\begin{eqnarray}\label{weak_form}
a(u,v)+(f(x,u),v)=(g,v),\quad \forall v\in V,
\end{eqnarray}
where
\begin{eqnarray}
a(u,v)=(\mathcal{A}\nabla u,\nabla v).
\end{eqnarray}
Obviously, $a(u,v)$ is bounded and coercive on $V$, i.e.,
\begin{eqnarray}\label{coercive}
a(u,v)\leq C_a\|u\|_{1,\Omega}\|v\|_{1,\Omega}\ \
\text{ and } \ \ c_a\|u\|_{1,\Omega}^2\leq a(u,u), \quad \forall u,v\in V.
\end{eqnarray}
Then we use the norm $\|w\|_a:=\sqrt{a(w,w)}$ for any $w\in V$ in this paper to replace the standard norm $\|\cdot\|_1$.

In order to guarantee the existence and uniqueness of the problem (\ref{weak_form}), we assume the nonlinear term
$f(\cdot,\cdot)$ satisfy the following assumption.

\textbf{Assumption A}: The nonlinear function $f(x,\cdot)$ satisfies the convex and Lipschitz continuous
conditions as follows
\begin{equation}\label{Assume Nonlinear}
\left\{
\begin{array}{l}
(f(x,w)-f(x,v),w-v)\geq 0,\ \ \ \forall w\in V,\ \forall v\in V,\\
(f(x,w)-f(x,v),\phi)\leq C_f\|w-v\|_0\|\phi\|_1,\ \ \ \forall w\in V,\ \forall v\in V,\ \forall\phi\in V.
\end{array}
\right.
\end{equation}

Now, we introduce the finite element method for semilinear elliptic
problem (\ref{weak_form}).
First we generate a shape regular decomposition of the computing
domain $\Omega\subset \cR^d \ (d=2, 3)$ into triangles or rectangles for
$d=2$, tetrahedrons or hexahedrons for $d=3$ (cf. \cite{BrennerScott, Ciarlet}).
The mesh diameter $h$ describes the maximum
diameter of all cells $K \in \cT_h$. Based on the mesh $\cT_h$,
we construct the finite element space $V_h \subset V$.  For simplicity, we set $V_h$ as the linear finite
 element space which is defined as follows
\begin{equation}\label{linear_fe_space}
  V_h = \big\{ v_h \in C(\Omega)\ \big|\ v_h|_{K} \in \mathcal{P}_1,
  \ \ \forall K \in \mathcal{T}_h\big\}\cap H^1_0(\Omega),
\end{equation}
where $\mathcal{P}_1$ denotes the linear function space.

The standard finite element scheme for semilinear equation (\ref{weak_form}) is:
Find $\bar{u}_h\in V_h$ such that
\begin{eqnarray}\label{FEM_form}
a(\bar{u}_h,v_h)+(f(x,\bar{u}_h),v_h)=(g,v_h), \quad \forall v_h\in V_h.
\end{eqnarray}

Denote a linearized operator $L:H_0^1(\Omega)\rightarrow H^{-1}(\Omega)$ by:
$$(Lw,v)= (\mathcal A\nabla w,\nabla v),\ \  \forall w\in V, \ \forall v\in V.$$
In order to deduce the global prior error estimates, we introduce $\eta_a(V_h)$ as follows:
\begin{eqnarray*}
\eta_{a}(V_h)=\sup_{f \in L^2(\Omega),\|f\|_0=1}\inf_{v_h \in V_h}\| L^{-1}f-v_h\|_a.
\end{eqnarray*}
It is easy to know that $\eta_a(V_h)\rightarrow 0$ as $h\rightarrow 0$
(cf. \cite{BrennerScott,Ciarlet}).

In order to measure the error for the finite element approximations, we denote
$$\delta_h(u)=\inf_{v_h\in V_h}\| u-v_h\|_a.$$
From \cite{shaidurov1995multigrid}, we can give the following error estimates.
\begin{lemma}\label{FEM_Estimate}
When Assumption A is satisfied, equations (\ref{weak_form})
and (\ref{FEM_form}) are uniquely solvable and the following estimates hold
\begin{eqnarray}
\|u-\bar u_h\|_a&\leq&(1+C\eta_a(V_h))\delta_h(u),\label{Super_Closes_Result}\\
\| u-\bar{u}_h\|_0&\lesssim&\eta_a(V_h)\| u-\bar{u}_h\|_a.\label{Direct_Estimate_0}
\end{eqnarray}
\end{lemma}
\begin{proof}
From Theorem 6.1 in \cite{shaidurov1995multigrid}, we can know that
problems (\ref{weak_form}) and (\ref{FEM_form})
are uniquely solvable. Now, it is time to prove the error estimates.
For this aim, we define the finite element projection operator $P_h$ by the following equation
\begin{eqnarray*}
a(P_hw,v_h)=a(w,v_h),\ \ \ \ \forall w\in V, \ \forall v_h\in V_h.
\end{eqnarray*}
It is easy to know that $\|u-P_hu\|_a=\delta_h(u)$ and $\|u-P_hu\|_0\lesssim \eta_a(V_h)\|u-P_hu\|_a$.
Let us define $w_h=P_hu-\bar u_h$ in this proof.
From (\ref{weak_form}), (\ref{Assume Nonlinear}) and (\ref{FEM_form}), we have
\begin{eqnarray*}
a(P_hu-\bar u_h,w_h)&\leq&a(P_hu-\bar u_h,w_h) + (f(x,P_hu)-f(x,\bar u_h),w_h)\nonumber\\
&=&a(P_hu,w_h)+(f(x,P_hu),w_h)-(g,w_h)\nonumber\\
&=&a(P_hu-u,w_h)+(f(x,P_hu)-f(x,u),w_h)\nonumber\\
&=&(f(x,P_hu)-f(x,u),w_h)\nonumber\\
&\leq& C_f\|u- P_hu\|_0\|w_h\|_a.
\end{eqnarray*}
Then the following inequalities hold
\begin{eqnarray}\label{Inequality_1}
\|P_hu-\bar u_h\|_a\leq C_f\|u-P_hu\|_0\leq C_f\eta_a(V_h)\|u-P_hu\|_a.
\end{eqnarray}
Combining (\ref{Inequality_1}) and the triangle inequality leads to the following estimates
\begin{eqnarray}
\|u-\bar u_h\|_a &\leq& \|u-P_hu\|_a+\|P_hu-\bar u_h\|_a\nonumber\\
&\leq& \delta_h(u)+C_f\eta_a(V_h)\|u-P_hu\|_a\nonumber\\
&\leq& (1+C_f\eta_a(V_h))\delta_h(u),
\end{eqnarray}
which is the desired result (\ref{Super_Closes_Result}).
From (\ref{Inequality_1}) and the triangle inequality, we have
\begin{eqnarray*}
\|u-\bar u_h\|_0&\leq& \|u-P_hu\|_0+\|P_hu-\bar u_h\|_0
\leq\|u-P_hu\|_0+ C\|P_hu-\bar u_h\|_a\nonumber\\
&\leq&C\eta_a(V_h)\|u-P_hu\|_a+C_f\eta_a(V_h)\|u-P_hu\|_a\nonumber\\
&\leq&(C+C_f)\eta_a(V_h)\|u-P_hu\|_a\leq (C+C_f)\eta_a(V_h)\|u-\bar u_h\|_a.
\end{eqnarray*}
This is the desired result (\ref{Direct_Estimate_0}) and the proof is complete.
\end{proof}


\section{Full multigrid method for semilinear elliptic equation}
In this section, a full multigrid method for
semilinear problems is proposed based on multilevel correction scheme in \cite{Xie_IMA,Xie_JCP}.
The key point is to transform the solution of the semilinear problem into a series of
solutions of the corresponding linear boundary value problems on the sequence of finite element
spaces and semilinear  problems on a very low dimensional space.
In order to carry out the multigrid method,
we first generate a coarse mesh $\mathcal{T}_H$
with the mesh size $H$ and the linear finite element space $V_H$ is
defined on the mesh $\mathcal{T}_H$. Then a sequence of
triangulations $\mathcal{T}_{h_k}$
of $\Omega\subset \mathcal{R}^d$ is determined as follows.
Suppose $\mathcal{T}_{h_1}$ (produced from $\mathcal{T}_H$ by
regular refinements) is given and let $\mathcal{T}_{h_k}$ be obtained
from $\mathcal{T}_{h_{k-1}}$ via one regular refinement step
(produce $\beta^d$ subelements) such that
\begin{eqnarray}\label{mesh_size_recur}
h_k=\frac{1}{\beta}h_{k-1},\ \ \ \ k=2,\cdots,n,
\end{eqnarray}
where the positive number $\beta$ denotes the refinement index and
larger than $1$ (always equals $2$).
Based on this sequence of meshes, we construct the corresponding
nested linear finite element spaces such that
\begin{eqnarray}\label{FEM_Space_Series}
V_{H}\subseteq V_{h_1}\subset V_{h_2}\subset\cdots\subset V_{h_n}.
\end{eqnarray}
Due to the convexity of the domain $\Omega$, the sequence of finite element spaces
$V_{h_1}\subset V_{h_2}\subset\cdots\subset V_{h_n}$
and the finite element space $V_H$ have  the following relations
of approximation accuracy
\begin{eqnarray}\label{delta_recur_relation}
\eta_a(V_{h_k})\approx \frac{1}{\beta}\eta_a(V_{h_{k-1}}),\ \ \ \
\delta_{h_k}(u)\approx \frac{1}{\beta}\delta_{h_{k-1}}(u),\ \ \ k=2,\cdots,n.
\end{eqnarray}
\subsection{One correction step}
In order to design the full multigrid method, first we introduce one correction step in this subsection.

Assume we have obtained an approximate solution $u_{h_k}^{(\ell)}\in V_{h_k}$.
A correction step to improve the accuracy of the
given approximation $u_{h_k}^{(\ell)}$ is designed as follows.
\begin{algorithm}\label{one step}
One Correction Step
\begin{enumerate}
\item Define the following auxiliary boundary value problem:
Find $\widehat{u}_{h_k}^{(\ell+1)} \in V_{h_k}$
such that
\begin{equation}\label{linearized}
a(\widehat{u}_{h_k}^{(\ell+1)},v_{h_k})
=-(f(x,u_{h_k}^{(\ell)}),v_{h_k})+(g,v_{h_k}), \quad \forall v_{h_k}\in V_{h_k}.
\end{equation}
Perform $m$ multigrid iteration steps for the second order elliptic equation to obtain an approximate
solution $\widetilde{u}_{h_k}^{(\ell+1)}$ with the following error reduction rate
\begin{eqnarray}\label{Error_Reduction}
\|\widetilde{u}_{h_k}^{(\ell+1)}-\widehat{u}_{h_k}^{(\ell+1)}\|_a
\leq \theta \|u_{h_k}^{(\ell)}-\widehat u_{h_k}^{(\ell+1)}\|_a,
\end{eqnarray}
where $u_{h_k}^{(\ell)}$ is used as the initial value for the multigrid iteration and
$\theta<1$ is a fixed constant independent from the mesh size $h_k$.
\item Define a finite element space $V_{H,h_k}:=V_H+{\rm span}\{\widetilde{u}_{h_k}^{\ell+1}\}$
and solve the following semilinear elliptic equation: Find $u_{h_k}^{(\ell+1)}\in V_{H,h_k}$ such that
\begin{equation}\label{coarse problem}
a(u_{h_k}^{(\ell+1)},v_{H,h_k})+(f(x,u_{h_k}^{(\ell+1)}),v_{H,h_k})=(g,v_{H,h_k}), \ \forall v_{H,h_k}\in V_{H,h_k}.
\end{equation}
\end{enumerate}
In order to simplify the notation and summarize the above two steps, we define
\begin{eqnarray*}
u_{h_k}^{(\ell+1)}={\rm SemilinearMG}(V_H, u_{h_k}^{(\ell)},V_{h_k}).
\end{eqnarray*}
\end{algorithm}
The error estimate of Algorithm \ref{one step} is studied in the following theorem.
\begin{theorem}\label{one step thm}
Assume the given solution $u_{h_k}^{(\ell)}$ has the following estimate
\begin{eqnarray}
\|\bar{u}_{h_k}-u_{h_k}^{(\ell)}\|_0\lesssim \eta_{a}(V_H)\|\bar{u}_{h_k}-u_{h_k}^{(\ell)}\|_a.
\end{eqnarray}
After the one correction step defined by Algorithm \ref{one step},
the resultant approximate solution $u_{h_k}^{(\ell+1)}$ has the following estimates
\begin{eqnarray}
\| \bar{u}_{h_k}-u_{h_k}^{(\ell+1)}\|_a&\leq& \gamma \| \bar{u}_{h_k}-u_{h_k}^{(\ell)}\|_a,\label{Error_a_k_ell+1}\\
\| \bar{u}_{h_k}-u_{h_k}^{(\ell+1)} \|_0&\leq& C\eta_a(V_H)\| \bar{u}_{h_k}-u_{h_k}^{(\ell+1)}\|_a,\label{Error_0_k_ell+1}
\end{eqnarray}
where $$ \gamma:= \big(\theta+(1+\theta)C\eta_a(V_H)\big)\big(1+C\eta_a(V_H)\big).$$
\end{theorem}
\begin{proof}
From (\ref{Assume Nonlinear}), (\ref{FEM_form}) and (\ref{linearized}), we have
\begin{eqnarray}\label{one step estimate}
&&a(\bar u_{h_k}-\widehat{u}_{h_k}^{(\ell+1)},v_{h_k})
=(f(x,u_{h_k}^{(\ell)})-f(x,\bar u_{h_k}),v_{h_k})\nonumber\\
&&\leq C_f\|\bar{u}_{h_k}-u_{h_k}^{(\ell)}\|_0\|v_{h_k}\|_a
\leq C\eta_a(V_H)\|\bar u_{h_k}-u_{h_k}^{(\ell)}\|_a\|v_{h_k}\|_a,\ \forall v_{h_k}\in V_{h_k}.
\end{eqnarray}
Combing (\ref{coercive}) and (\ref{one step estimate}) leads to
\begin{eqnarray}\label{Inequality_2}
\| \bar u_{h_k}-\widehat{u}_{h_k}^{(\ell+1)} \|_a \leq C\eta_a(V_H)\|\bar u_{h_k}-u_{h_k}^{(\ell)}\|_a.
\end{eqnarray}
After performing $m$ multigrid iteration steps, from (\ref{Error_Reduction}) and (\ref{Inequality_2}),
the following estimates hold
\begin{eqnarray}
\|\widetilde u_{h_k}^{(\ell+1)}-\bar u_{h_k}\|_a
&\leq& \|\widetilde u_{h_k}^{(\ell+1)}-\widehat u_{h_k}^{(\ell+1)}\|_a
+\|\widehat u_{h_k}^{(\ell+1)}-\bar u_{h_k}\|_a\nonumber\\
&\leq& \theta\|u_{h_k}^{(\ell)}-\widehat u_{h_k}^{(\ell+1)}\|_a
+\|\widehat u_{h_k}^{(\ell+1)}-\bar u_{h_k}\|_a\nonumber\\
&\leq& \theta\|u_{h_k}^{(\ell)}-\bar u_{h_k}\|_a+\theta\|\widehat u_{h_k}^{(\ell+1)}-\bar u_{h_k}\|_a
+\|\widehat u_{h_k}^{(\ell+1)}-\bar u_{h_k}\|_a\nonumber\\
&\leq& \big(\theta+(1+\theta)C\eta_a(V_H)\big)\|\bar u_{h_k}-u_{h_k}^{(\ell)}\|_a.
\end{eqnarray}
Note that the semilinear elliptic problem (\ref{coarse problem}) can be regarded as a finite dimensional
approximation of the semilinear elliptic problem (\ref{FEM_form}).
Let $P_{H,h_k}: V\rightarrow V_{H,h_k}$ denotes the finite element projection operator which is defined as follows
\begin{eqnarray*}
a(P_{H,h_k}w,v_{H,h_k})&=&a(w,v_{H,h_k}),\ \ \ \ \forall w\in V,\ \forall v_{H,h_k}\in V_{H,h_k}.
\end{eqnarray*}
Since $\widetilde u_{h_k}^{(\ell+1)}\in V_{H,h_k}$ and $V_H\subset V_{H,h_k}$, it is obvious that
$\eta_a(V_{H,h_k})\leq \eta_a(V_H)$ and
\begin{eqnarray}
\|\bar u_{h_k}-P_{H,h_k}\bar u_{h_k}\|_a
&=& \inf_{v_{H,h_k}\in V_{H,h_k}}\|\bar u_{h_k}-v_{H,h_k}\|_a \nonumber\\
&\leq& \|\bar u_{h_k}-\widetilde u_{h_k}^{(\ell+1)}\|_a,\label{Inequality_4}\\
\|\bar u_{h_k}-P_{H,h_k}\bar u_{h_k}\|_0
&\leq& C\eta_a(V_{H,h_k})\|\bar u_{h_k}-P_{H,h_k}\bar u_{h_k}\|_a\nonumber\\
&\leq& C\eta_a(V_H)\|\bar u_{h_k}-P_{H,h_k}\bar u_{h_k}\|_a.\label{Inequality_5}
\end{eqnarray}
Let us define $w_{h_k}=P_{H,h_k}\bar u_{h_k}- u_{h_k}^{(\ell+1)}\in V_{H,h_k}$ in this proof.
Based on problems (\ref{FEM_form}) and (\ref{coarse problem}),  the following estimates hold
\begin{eqnarray}\label{Inequality_3}
&&a(P_{H,h_k}\bar u_{h_k}- u_{h_k}^{(\ell+1)},w_{h_k})\nonumber\\
&\leq&a(P_{H,h_k}\bar u_{h_k}- u_{h_k}^{(\ell+1)},w_{h_k})
+ (f(x,P_{H,h_k}\bar u_{h_k})-f(x,u_{h_k}^{(\ell+1)}),w_{h_k})\nonumber\\
&=&a(P_{H,h_k}\bar u_{h_k},w_h)+(f(x,P_{H,h_k}\bar u_{h_k}),w_{h_k})-(g,w_{h_k})\nonumber\\
&=&a(P_{H,h_k}\bar u_{h_k}-\bar u_{h_k},w_{h_k})+(f(x,P_{H,h_k}\bar u_{h_k})-f(x,\bar u_{h_k}),w_{h_k})\nonumber\\
&=&(f(x,P_{H,h_k}\bar u_{h_k})-f(x,\bar u_{h_k}),w_{h_k})
\leq C_f\|\bar u_{h_k}- P_{H,h_k}\bar u_{h_k}\|_0\|w_{h_k}\|_a.
\end{eqnarray}
From (\ref{Inequality_5}) and (\ref{Inequality_3}), we have
\begin{eqnarray}\label{Inequality_6}
\|P_{H,h_k}\bar u_{h_k}- u_{h_k}^{(\ell+1)}\|_a &\leq& C_f\|\bar u_{h_k}- P_{H,h_k}\bar u_{h_k}\|_0\nonumber\\
&\leq&C\eta_a(V_H)\|\bar u_{h_k}- P_{H,h_k}\bar u_{h_k}\|_a.
\end{eqnarray}
Combining (\ref{Inequality_4}), (\ref{Inequality_6}) and triangle inequality leads to the following inequalities
\begin{eqnarray}
\|\bar u_{h_k}-u_{h_k}^{(\ell+1)}\|_a&\leq& \|\bar u_{h_k}-P_{H,h_k}\bar u_{h_k}\|_a+\|P_{H,h_k}\bar u_{h_k}-u_{h_k}^{(\ell+1)}\|_a\nonumber\\
&\leq&(1+C\eta_a(V_H))\|\bar u_{h_k}-P_{H,h_k}\bar u_{h_k}\|_a\nonumber\\
&\leq&(1+C\eta_a(V_H))\|\bar u_{h_k}-\widetilde u_{h_k}^{(\ell+1)}\|_a.
\end{eqnarray}
This is the desired result (\ref{Error_a_k_ell+1}).
From (\ref{Inequality_3}) and the triangle inequality, we have the following estimates
\begin{eqnarray}
\|\bar u_{h_k}-u_{h_k}^{(\ell+1)}\|_0&\leq& \|\bar u_{h_k}-P_{H,h_k}\bar u_{h_k}\|_0+\|P_{H,h_k}\bar u_{h_k}-u_{h_k}^{(\ell+1)}\|_0\nonumber\\
&\leq& \|\bar u_{h_k}-P_{H,h_k}\bar u_{h_k}\|_0+C\|P_{H,h_k}\bar u_{h_k}-u_{h_k}^{(\ell+1)}\|_a\nonumber\\
&\leq&C\eta_a(V_H)\|\bar u_{h_k}-P_{H,h_k}\bar u_{h_k}\|_a\nonumber\\
&\leq&C\eta_a(V_H)\|\bar u_{h_k}-u_{h_k}^{(\ell+1)}\|_a,
\end{eqnarray}
which is the desired result (\ref{Error_0_k_ell+1}) and the proof is complete.
\end{proof}
\begin{remark}\label{Remark_1}
The proof of Theorem \ref{one step thm} shows that the structure of the low dimensional
space $V_{H,h_k}$ plays the key role for Algorithm \ref{one step}. This special space makes the
finite element projection $P_{H,h_k}$ has both the accuracy as in (\ref{Inequality_4}) and
the $L^2$-norm estimate by duality argument as in (\ref{Inequality_5}).
\end{remark}
\subsection{Full multigrid method}
In this subsection, a full multigrid method is proposed based on the one correction step
defined in Algorithm \ref{one step}. This algorithm can reach the optimal convergence rate with
the optimal computational complexity.
\begin{algorithm}\label{Multilevel Correction}
Full Multigrid Scheme
\begin{enumerate}
\item Solve the following semilinear problem in $V_{h_1}$:
Find $u_{h_1}\in V_{h_1}$ such that
\begin{equation*}
a(u_{h_1}, v_{h_1})+(f(x,u_{h_1}), v_{h_1})=(g,v_{h_1}), \quad \forall v_{h_1}\in  V_{h_1}.
\end{equation*}
\item For $k=2,\cdots,n$, do the following iteration:
\begin{enumerate}
\item Set $u_{h_k}^{(0)}=u_{h_{k-1}}$.
\item For $\ell=0,\cdots, p-1$, do the following iterations
\begin{eqnarray*}
u_{h_k}^{(\ell+1)}={\rm SemilinearMG}(V_H,u_{h_k}^{(\ell)},V_{h_k}).
\end{eqnarray*}
\item Define $u_{h_k}=u_{h_k}^{(p)}$.
\end{enumerate}
End Do
\end{enumerate}
Finally, we obtain an approximate solution $u_{h_n}\in V_{h_n}$.
\end{algorithm}
\begin{theorem}\label{Multilevel Correction Thm}
After implementing Algorithm \ref{Multilevel Correction}, we have the following error estimates
for the final approximation $u_{h_n}$
\begin{eqnarray}
\|\bar{u}_{h_n}-u_{h_n}\|_a &\leq&\frac{2\gamma^p\beta}{1-\gamma^p\beta}\delta_{h_n}(u),\label{Estimate1}\\
\|\bar{u}_{h_n}-u_{h_n}\|_0 &\leq& C\eta_a(V_H)\|\bar{u}_{h_n}-u_{h_n}\|_a, \label{Estimate2}
\end{eqnarray}
under the condition that the coarsest mesh size $H$ is small enough such that $\gamma^p\beta<1$.
\end{theorem}
\begin{proof}
From the first step of Algorithm \ref{Multilevel Correction}, we have $u_{h_1}=\overline{u}_{h_1}$.
Then from Lemma \ref{FEM_Estimate} and the proof of Theorem \ref{one step}, the following estimates hold
\begin{eqnarray}
\|\bar u_{h_2}-u_{h_2}\|_a&=&\|\bar u_{h_2}-u_{h_2}^{(p)}\|_a\leq \gamma^p\|\bar u_{h_2}-u_{h_2}^{(0)}\|_a\nonumber\\
&=&\gamma^p\|\bar u_{h_2}- u_{h_1}\|_a=\gamma^p\|\bar u_{h_2}- \bar u_{h_1}\|_a.\label{tmp0}\\
\|\overline{u}_{h_2}-u_{h_2}\|_0&\leq& C\eta_a(V_H)\|\bar u_{h_2}-u_{h_2}\|_a.\label{tmp1}
\end{eqnarray}
Based on (\ref{tmp0}), (\ref{tmp1}), Theorem \ref{one step} and recursive argument,
the final approximate solution has the following error estimates
\begin{eqnarray*}
\|\bar u_{h_n}-u_{h_n}\|_a &\leq& \gamma^p\|\bar u_{h_n}-u_{h_n}^{(0)}\|_a=\gamma^p\|\bar u_{h_n}-u_{h_{n-1}}\|_a\nonumber\\
&\leq& \gamma^p\big(\|\bar u_{h_n}-\bar u_{h_{n-1}}\|_a + \|\bar u_{h_{n-1}}-u_{h_{n-1}}\|_a\big)\nonumber\\
&\leq& \gamma^p\|\bar u_{h_n}-\bar u_{h_{n-1}}\|_a + \gamma^{2p}\big(\|\bar u_{h_{n-1}}-\bar u_{h_{n-2}}\|_a
+\|\bar u_{h_{n-2}}- u_{h_{n-2}}\|_a\big)\nonumber\\
&\leq&\sum_{k=1}^{n-1}\gamma^{kp}\|\bar u_{h_{n-k+1}}-\bar u_{h_{n-k}}\|_a\nonumber\\
&\leq&\sum_{k=1}^{n-1}\gamma^{kp}\big(\|\bar u_{h_{n-k+1}}-u\|_a+\|u-\bar u_{h_{n-k}}\|_a\big)\nonumber\\
&\leq&2\sum_{k=1}^{n-1}\gamma^{kp}\delta_{h_{n-k}}(u)\leq 2\sum_{k=1}^{n-1}\gamma^{kp}\beta^k\delta_{h_n}(u)
\leq\frac{2\gamma^p\beta}{1-\gamma^p\beta}\delta_{h_n}(u),
\end{eqnarray*}
which is just the desired result (\ref{Estimate1}).
The second result (\ref{Estimate2}) can be proved by the similar argument in the proof
of Theorem \ref{one step thm} and the proof is complete.
\end{proof}
\begin{corollary}\label{corollary}
For the final approximation $u_{h_n}$ obtained by Algorithm \ref{Multilevel Correction}, we have the following estimates
\begin{eqnarray}
\|u-u_{h_n}\|_a&\lesssim& \delta_{h_n}(u),\\
\|u-u_{h_n}\|_0&\lesssim& \eta_a(V_{H})\delta_{h_n}(u).
\end{eqnarray}
\end{corollary}
\begin{proof}
This is a direct consequence of the combination of Lemma \ref{FEM_Estimate} and Theorem \ref{Multilevel Correction Thm}.
\end{proof}
\subsection{Estimate of the computational work}
In this subsection, we turn our attention to the estimate of computational work
for the full multigrid method defined in Algorithm \ref{Multilevel Correction}. It will be shown that
the full multigrid method makes solving the semilinear elliptic problem almost as cheap as
solving the corresponding linear boundary value problems.

First, we define the dimension of each level
finite element space as $N_k:={\rm dim}V_{h_k}$. Then we have
\begin{eqnarray}\label{relation_dimension}
N_k\approx\Big(\frac{1}{\beta}\Big)^{d(n-k)}N_n,\ \ \ k=1,2,\cdots, n.
\end{eqnarray}

The computational work for the second step in Algorithm \ref{Multilevel Correction} is different
from the linear elliptic problems \cite{BramblePasciak,ScottZhang,ToselliWidlund,shaidurov1995multigrid,Xu}.
In this step, we need to solve a semilinear elliptic problem (\ref{coarse problem}).
Always, some type of nonlinear iteration method (fixed-point iteration or
Newton type iteration) is adopted to solve this low dimensional semilinear elliptic problem.
In each nonlinear iteration step, it is required to assemble the matrix on the finite element
space $V_{H,h_k}$ ($k=2,\cdots,n$) which needs the computational work $\mathcal{O}(N_k)$.
Fortunately, the matrix assembling can be carried out by the parallel way easily
in the finite element space since it has no data transfer.

\begin{theorem}\label{optimal_work}
Assume we use $\vartheta$ computing-nodes in Algorithm \ref{Multilevel Correction}, the semilinear
elliptic solving in the coarse spaces $V_{H,h_k}$ ($k=2,\cdots, n$) and
$V_{h_1}$ need work $\mathcal{O}(M_H)$ and $\mathcal{O}(M_{h_1})$, respectively,
and the work of the multigrid iteration for the boundary value problem
in each level space $V_{h_k}$ is $\mathcal{O}(N_k)$ for $k=2,3,\cdots,n$.
Let $\varpi$ denote the nonlinear iteration times when we solve
the semilinear elliptic problem (\ref{coarse problem}).
Then in each computational node, the work involved
in Algorithm \ref{Multilevel Correction} has the following estimate
\begin{eqnarray}\label{Computation_Work_Estimate}
{\rm Total\ work}&=&\mathcal{O}\Big(\big(1+\frac{\varpi}{\vartheta}\big)N_n
+ M_H\log N_n+M_{h_1}\Big).
\end{eqnarray}
\end{theorem}
\begin{proof}
We use $W_{k}$ to denote the work involved in each correction step on the $k$-th
finite element space $V_{h_{k}}$.
From the definition of Algorithm \ref{Multilevel Correction}, we have the following estimate
\begin{eqnarray}\label{work_k}
W_k&=&\mathcal{O}\left(N_k +M_H+\varpi\frac{N_k}{\vartheta}\right).
\end{eqnarray}
Based on the property (\ref{relation_dimension}), iterating (\ref{work_k}) leads to
\begin{eqnarray}\label{Work_Estimate}
\text{Total work} &=& \sum_{k=1}^nW_k\nonumber =
\mathcal{O}\left(M_{h_1}+\sum_{k=2}^n
\Big(N_k + M_H+\varpi\frac{N_k}{\vartheta}\Big)\right)\nonumber\\
&=& \mathcal{O}\left(\sum_{k=2}^n\Big(1+\frac{\varpi}{\vartheta}\Big)N_k
+ (n-1) M_H + M_{h_1}\right)\nonumber\\
&=& \mathcal{O}\left(\sum_{k=2}^n
\Big(\frac{1}{\beta}\Big)^{d(n-k)}\Big(1+\frac{\varpi}{\vartheta}\Big)N_n
+ M_H\log N_n+M_{h_1}\right)\nonumber\\
&=& \mathcal{O}\left(\big(1+\frac{\varpi}{\vartheta}\big)N_n
+ M_H\log N_n+M_{h_1}\right).
\end{eqnarray}
This is the desired result and we complete the proof.
\end{proof}
\begin{remark}
Since we always have a good enough initial solution $\widetilde{u}_{h_k}^{(\ell+1)}$
in the second step of Algorithm \ref{one step},
then solving the semilinear elliptic problem (\ref{coarse problem}) never
needs many nonlinear iterations.
In this case, the complexity in each computational node will be $\mathcal{O}(N_n)$ provided
$M_H\ll N_n$ and $M_{h_1}\leq N_n$. For more difficult nonlinear problems, the complexity
in each computational node can also be bounded to $\mathcal{O}(N_n)$ by the parallel way with enough
computational nodes.
\end{remark}

\section{Numerical results}
In this section, four numerical experiments are presented to verify the
theoretical analysis and efficiency of Algorithm \ref{Multilevel Correction}.
We will check different nonlinear terms which include polynomial, exponential functions
and a function only having bounded first order derivative. Furthermore, we also
investigate the performance of the full multigrid method on the adaptively refined
meshes. In all examples, we choose $m=2$ and $p=1$.

\subsection{Example 1}
We consider the following semilinear elliptic problem:
\begin{equation}\label{semilinear equation example1}
\left\{
\begin{array}{rcl}
-\Delta u+u^3&=& g, \quad \text{in } \Omega,\\
u&=& 0, \quad \text{on } {\partial\Omega},
\end{array}
\right.
\end{equation}
where $\Omega=(0,1)^3$. We choose the right hand side term $g$
such that the exact solution is given by
\begin{eqnarray}
u = \sin(\pi x)\sin(\pi y) \sin(\pi z).
\end{eqnarray}

\begin{figure}[htb]
\centering
\includegraphics[width=7cm]{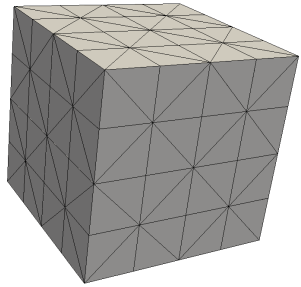}
\caption{\small\texttt The initial mesh for Example 1}
\label{Initial_Cube}
\end{figure}

We give the numerical results for the approximate solutions by
Algorithm \ref{Multilevel Correction}.
Figure \ref{Initial_Cube} shows the initial triangulation.
Figure \ref{error0} shows the error estimates and the CPU time in seconds.
It is shown in  Figure \ref{error0} that the
 approximate solution by Algorithm \ref{Multilevel Correction}
has the optimal convergence order and the linear computational complexity
which coincides with the theoretical results in Theorems \ref{one step thm},
\ref{Multilevel Correction Thm} and Corollary \ref{corollary}.

\begin{figure}[htb]
\centering
\includegraphics[width=6cm]{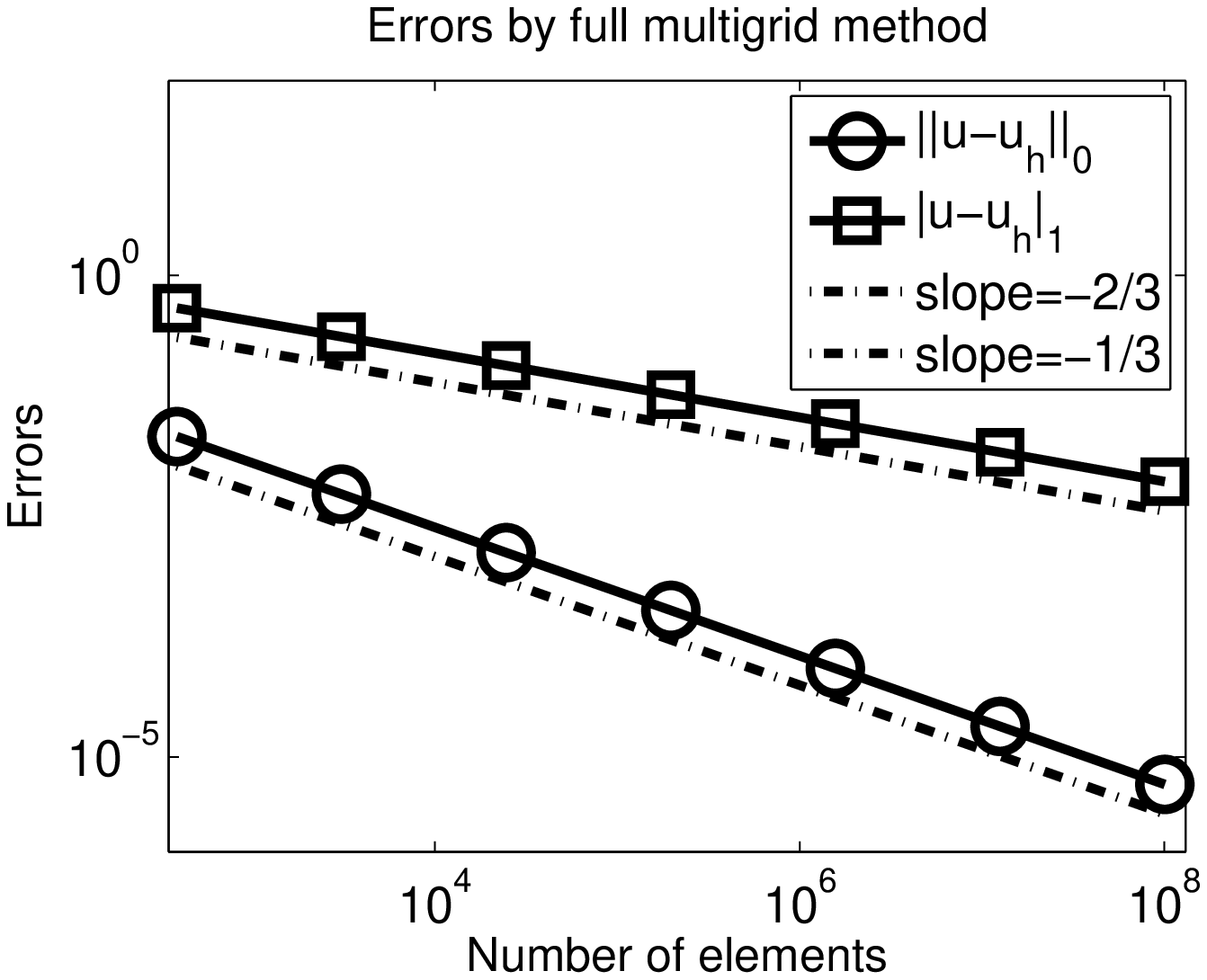}
\includegraphics[width=6cm]{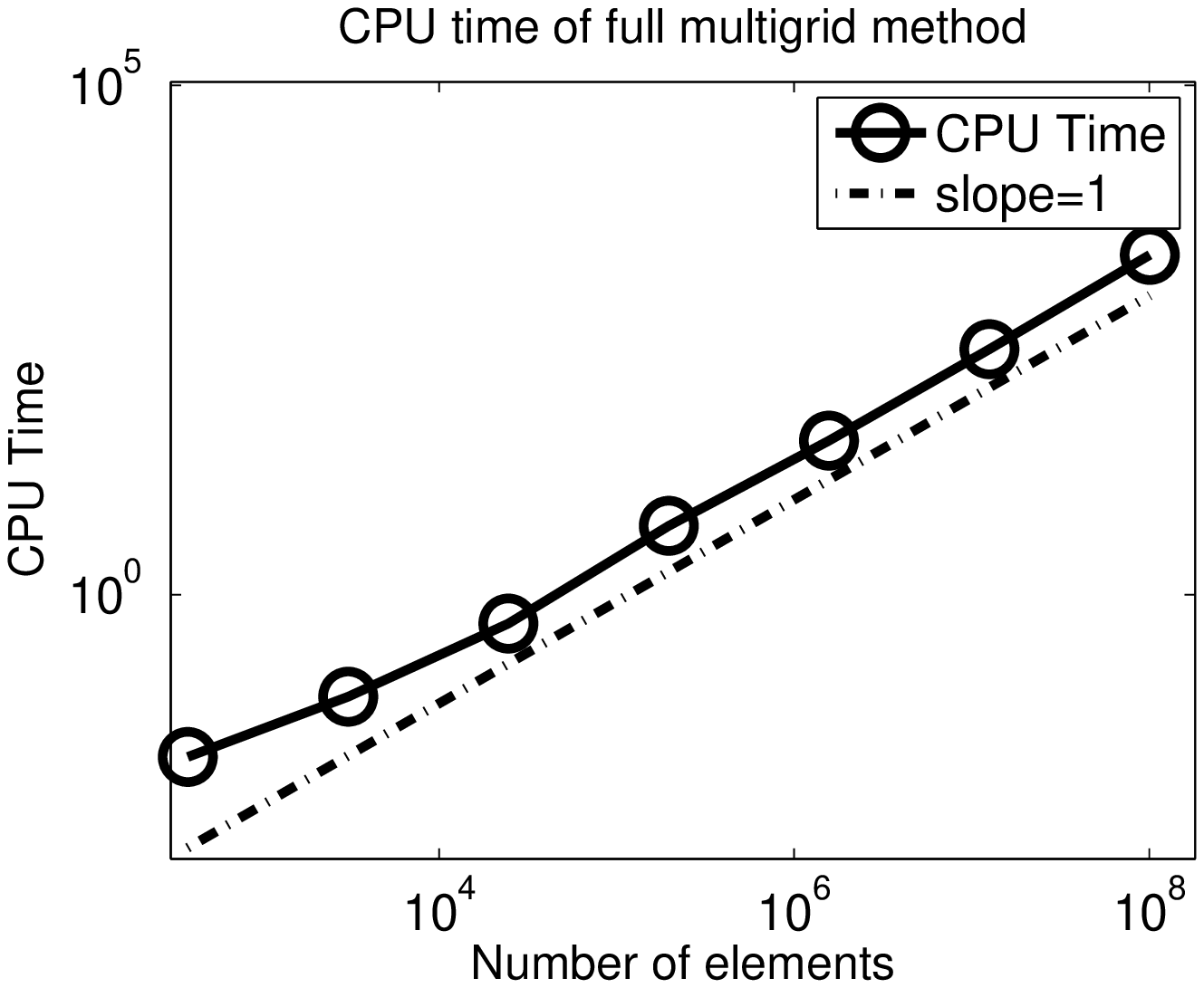}
\caption{Errors and CPU time (in seconds) of Algorithm \ref{Multilevel Correction}  for Example 1}
\label{error0}
\end{figure}
\subsection{Example 2}
In the second example, we solve the following semilinear elliptic problem:
\begin{equation}\label{semilinear equation example2}
\left\{
\begin{array}{rcl}
-\Delta u-e^{-u}&=& 1, \quad \text{in } \Omega,\\
u&=& 0, \quad \text{on } {\partial\Omega},
\end{array}
\right.
\end{equation}
where $\Omega=(0,1)^3$.
Since the exact solution is not known, we choose an adequate accurate
approximate solution on a fine enough mesh as the exact one.

Algorithm \ref{Multilevel Correction} is applied to this example.
Figure \ref{Initial_Cube} shows the initial mesh.
Figure \ref{error1} gives the corresponding
numerical results which also show
the optimal convergence rate and linear computational complexity of Algorithm \ref{Multilevel Correction}.

\begin{figure}[htb]
\centering
\includegraphics[width=6cm]{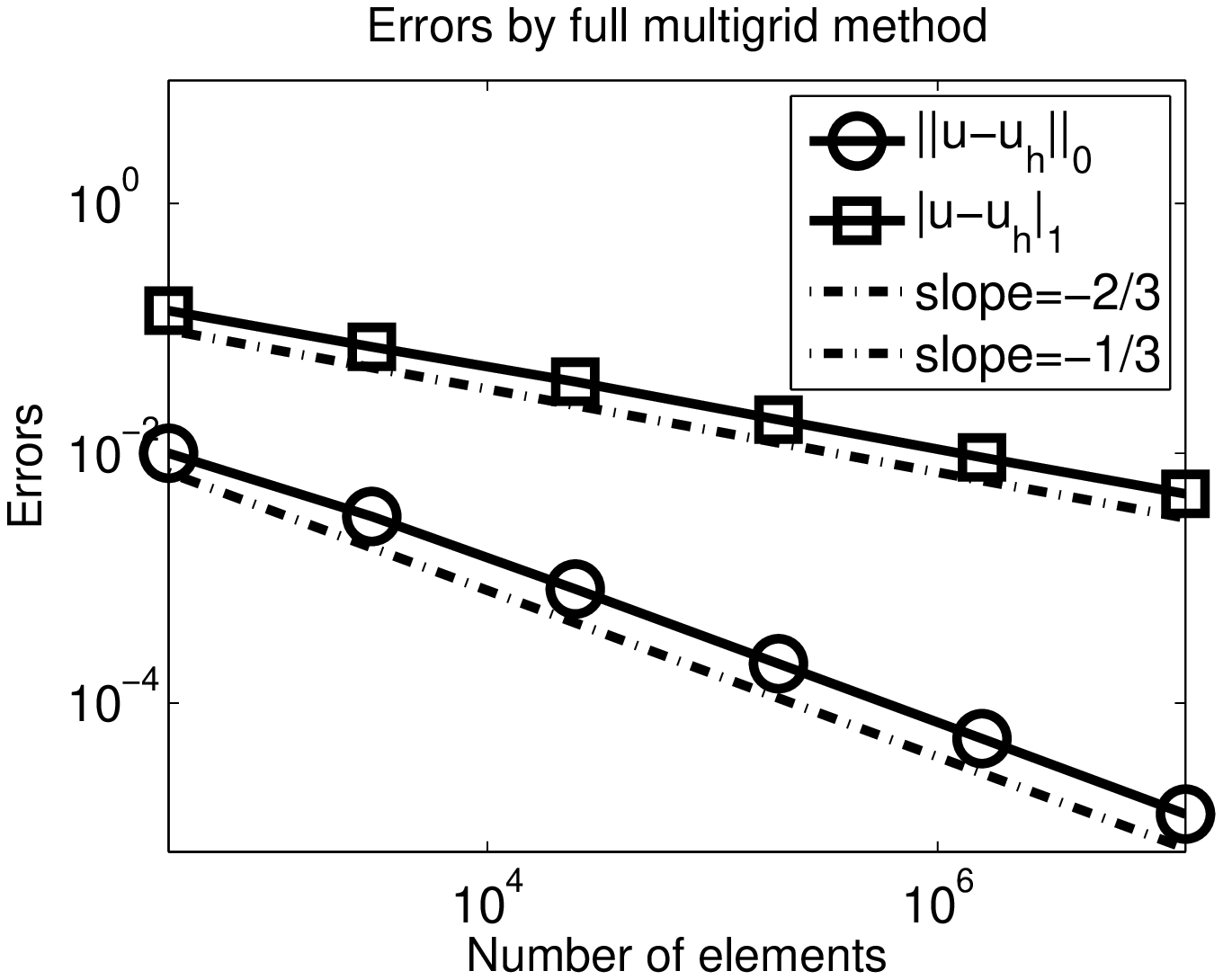}
\includegraphics[width=6cm]{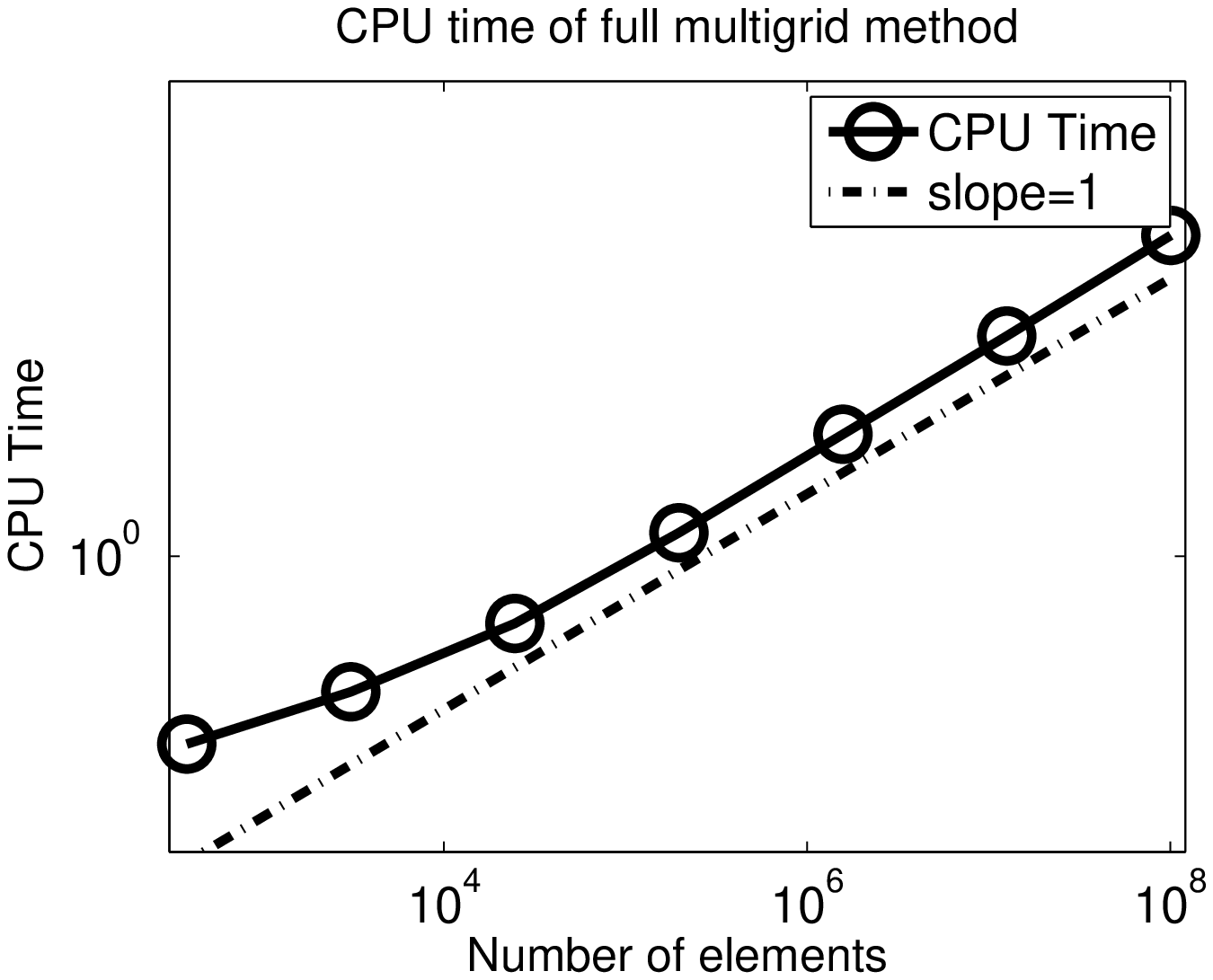}
\caption{Errors and CPU time (in seconds) of Algorithm \ref{Multilevel Correction}  for Example 2}
\label{error1}
\end{figure}
\subsection{Example 3}
In the third example, we solve the following semilinear elliptic problem:
\begin{equation}\label{semilinear equation example3}
\left\{
\begin{array}{rcl}
-\Delta u+f(x,u)&=& g, \quad \text{in } \Omega,\\
u&=& 0, \quad \text{on } {\partial\Omega},
\end{array}
\right.
\end{equation}
with
\begin{equation}
f(x,u)=
\left\{
\begin{array}{rcl}
u^{3/2}, \quad  &\text{if}& u\geq 0,\\
-u^{3/2},\quad &\text{if}& u<0,
\end{array}
\right.
\end{equation}
where $\Omega=(0,1)^3$.
We choose the right hand side term $g$ such that the exact solution is given by
\begin{eqnarray}
u = \sin(2\pi x)\sin(2\pi y) \sin(2\pi z).
\end{eqnarray}
In this example, the nonlinear term $f(x,v)$ has bounded first order derivative $\partial f(x,v)/\partial v$
but unbounded second order derivative $\partial^2 f(x,v)/\partial^2 v$.  Then the
methods given in \cite{HuangShiTangXue,shaidurov1995multigrid} can not be used for this example.

Algorithm \ref{Multilevel Correction} is applied to this example.
Figure \ref{Initial_Cube} shows the initial mesh. Figure \ref{error3} gives the corresponding
numerical results which also show the optimal convergence rate and linear
computational complexity of Algorithm \ref{Multilevel Correction}.

\begin{figure}[htb]
\centering
\includegraphics[width=6cm]{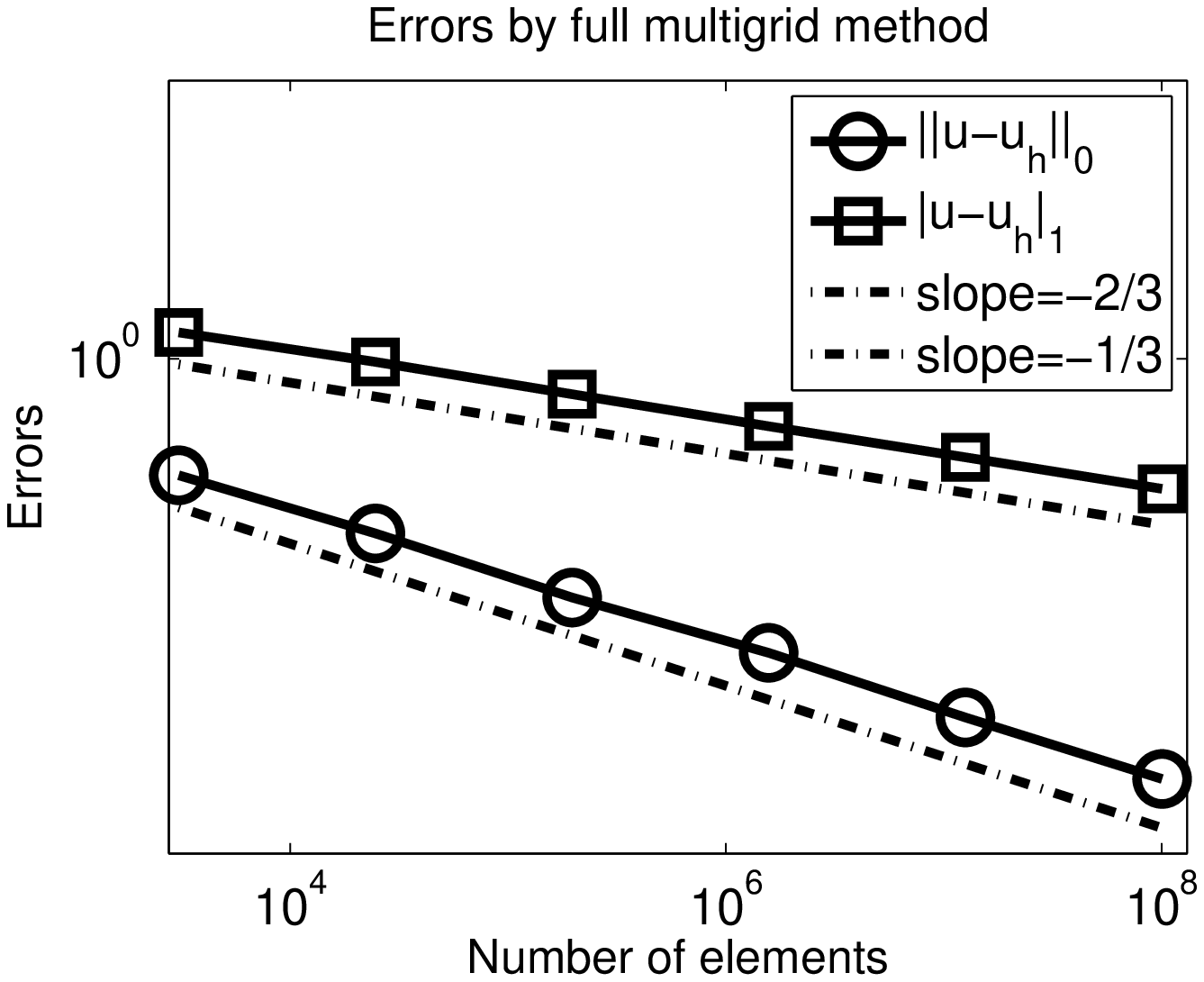}
\includegraphics[width=6cm]{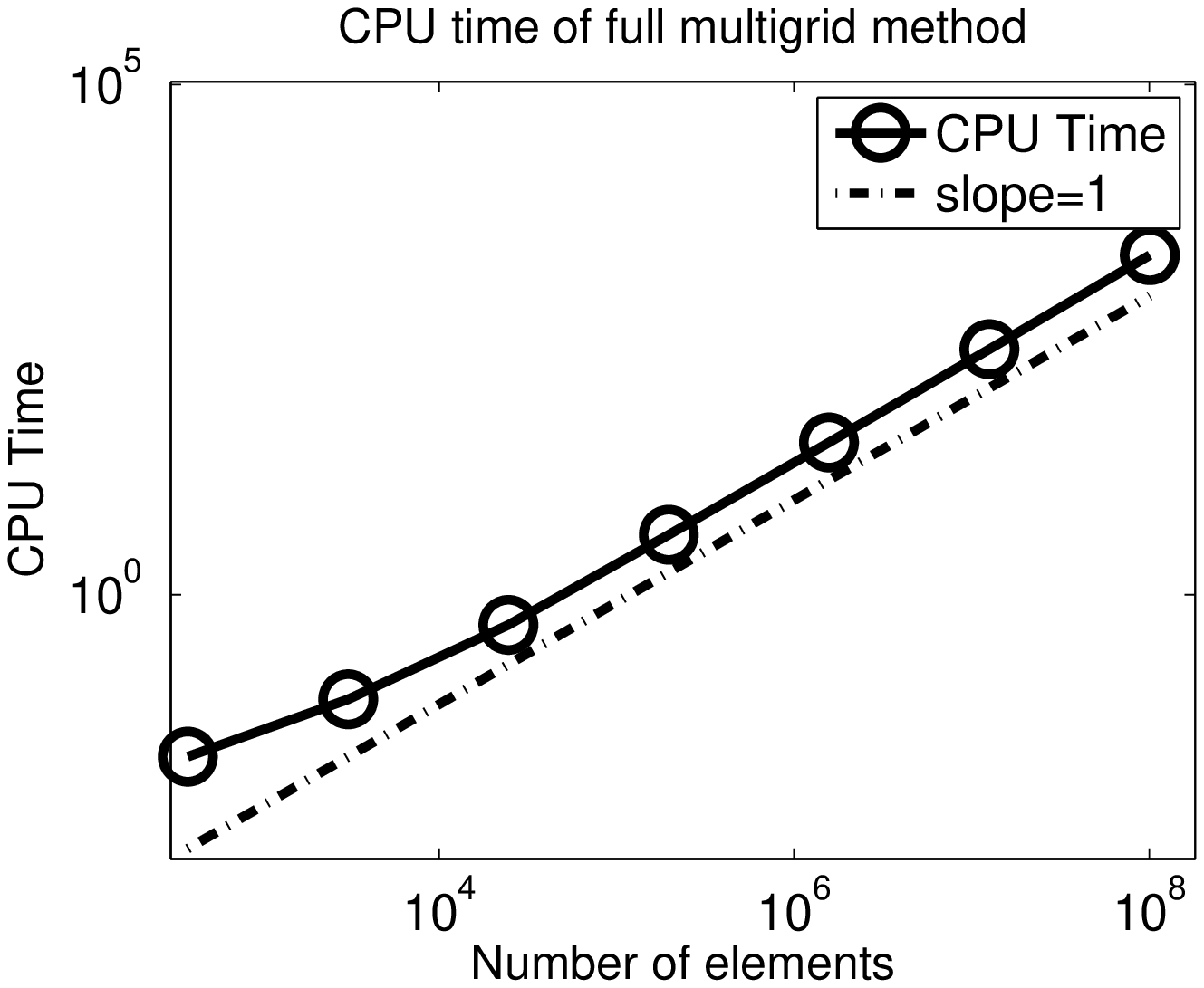}
\caption{Errors and CPU time (in seconds) of Algorithm \ref{Multilevel Correction}  for Example 3}
\label{error3}
\end{figure}

\subsection{Example 4}
In the last example, we solve the following semilinear elliptic problem:
\begin{equation}\label{semilinear equation example4}
\left\{
\begin{array}{rcl}
-\Delta u+u^{3/2}&=& 1 \quad \text{in } \Omega,\\
u&=& 0 \quad \text{on } {\partial\Omega},
\end{array}
\right.
\end{equation}
where $\Omega=(-1,1)^3\setminus [0,1)^3$.
Due to the reentrant corner of $\Omega$, the exact solution with singularities is
expected. The convergence order for approximate solution is less than the order predicted
by the theory for regular solutions.  Thus, the adaptive refinement is adopted to couple
with the full multigrid  method described in Algorithm \ref{Multilevel Correction} (cf. \cite{LXX}).

Since the exact solution is not known, we also choose an adequately accurate
approximation on a fine enough mesh as the exact one. We give the numerical
results of the full multigrid method in which the sequence of meshes
$\cT_{h_1},\cdots,\cT_{h_n}$ is produced by the adaptive refinement with the following a posteriori
error estimator
\begin{eqnarray}
\eta^2(v,K):=h_K^2\| {\cR}_{K}(v)\|_{0,K}^2
+\sum_{e\in\cE_I,e\subset \partial K}h_e\| {\cJ}_{e}(v)\|_{0,e}^2,
\end{eqnarray}
where  the element residual  ${\cR}_{K}(v)$ and the jump residual ${\cJ}_e(v)$ are defined as follows:
\begin{eqnarray}
&&{\cR}_{K}(v) := g-f(x,v)-\nabla\cdot(\mathcal{A}\nabla v), \qquad \text{in } K\in \cT_{h_k}, \\
&&{\cJ}_{e}(v) := -\mathcal{A}\nabla v^+\cdot\nu^+-\mathcal{A}\nabla v^-\cdot\nu^-
:=[\mathcal{A}\nabla v]_e\cdot \nu_e,  \quad\text{on } e\in \cE_I.
\end{eqnarray}
Here  $\cE_I$ denotes the set of interior faces (edges or sides) of $\cT_{h_k}$ and
$e$ is the common side of elements $K^+$ and $K^-$ with the unit outward normals $\nu^+$ and
$\nu^-$, respectively, and $ \nu_e=\nu^- $.

Figure \ref{adaptive mesh} shows the mesh after $15$ refinements and the corresponding cross section.
Figure \ref{error4} shows the numerical results by Algorithm \ref{Multilevel Correction}.
From Figure \ref{error4}, we can find that the full multigrid method can
also work on the adaptive family of meshes and obtain the optimal accuracy.
The full multigrid method can be coupled with the adaptive refinement naturally to produce
a type of adaptive finite element method for semilinear elliptic problem where the direct nonlinear iteration
in the adaptive finite element space is not required. This can also improve the overall efficiency of the adaptive
finite element method for semilinear elliptic problem solving. For more information, please refer to
the paper \cite{LXX}.

\begin{figure}[htb]
\centering
\includegraphics[width=5cm]{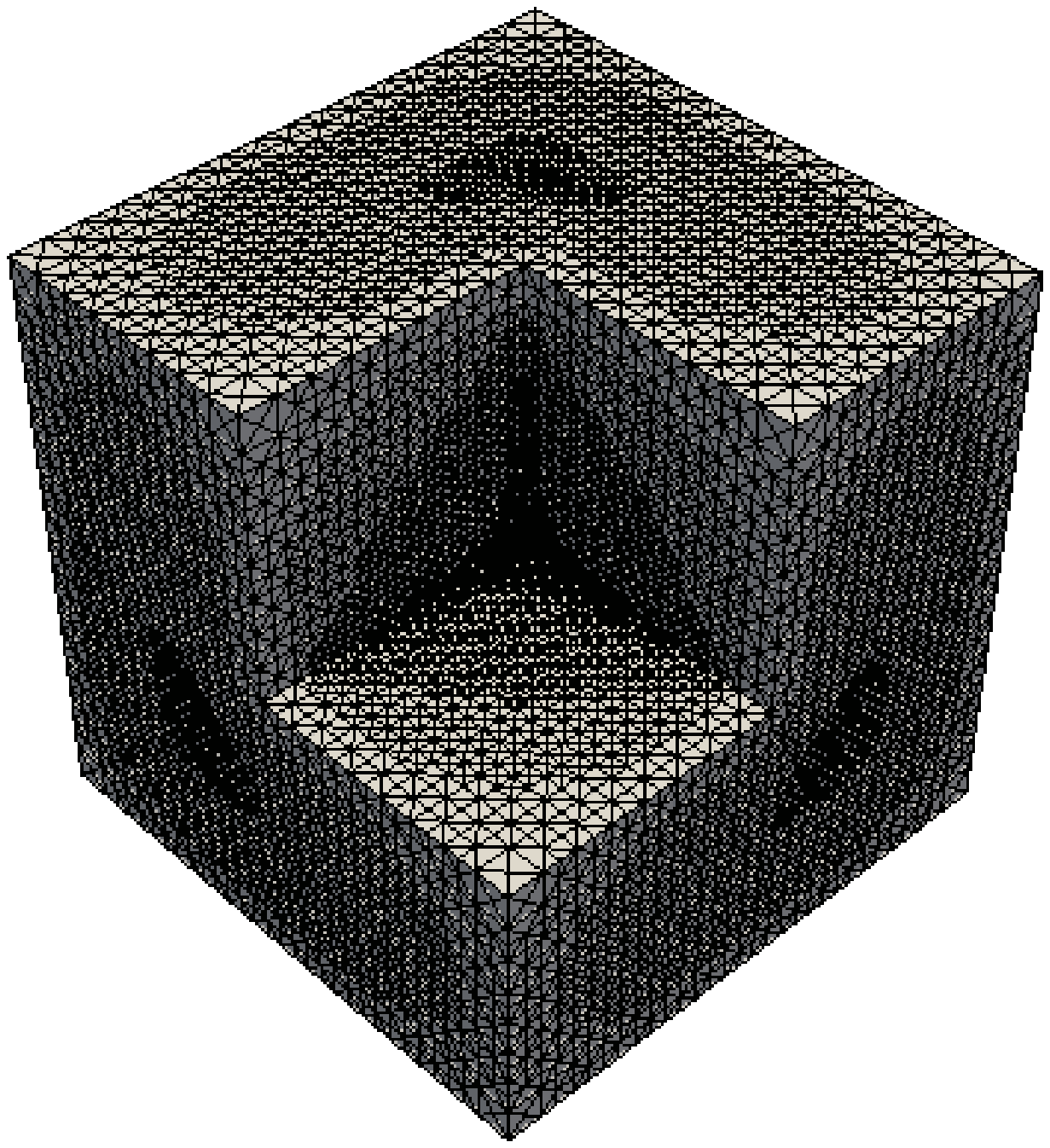} \ \ \ \ \
\includegraphics[width=5cm]{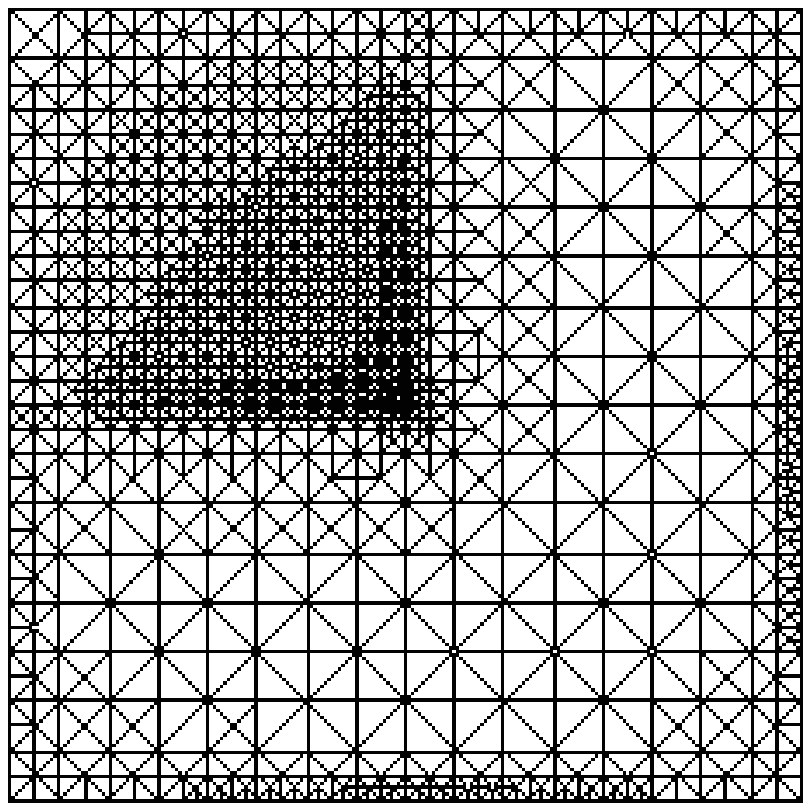}
\caption{The triangulations after $15$ adaptive refinements and the corresponding cross section for Example 4}
\label{adaptive mesh}
\end{figure}

\begin{figure}[htb]
\centering
\includegraphics[width=6cm]{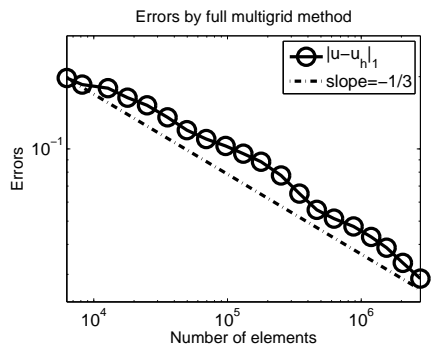}\label{ex3-e1}
\caption{Errors of Algorithm \ref{Multilevel Correction}  for Example 4}
\label{error4}
\end{figure}

\section{Concluding remarks}
In this paper, a full multigrid method is proposed for solving semilinear elliptic
equations by the finite element method. The corresponding estimates of error 
and computational work are given. The main idea is to transform the solution of 
the semilinear problem into a series of solutions of the corresponding linear 
boundary value problems on the sequence of finite element spaces and semilinear 
problems on a very low dimensional space. Compared with the existing multigrid 
methods which require the bounded second order derivatives of the nonlinear 
 term, the proposed method only needs the Lipschitz continuation in some sense of the nonlinear term.
Based on the full multigrid method, all existing efficient solvers for the linear elliptic problems can
serve as solvers for the semilinear equations. The idea and algorithm in this paper can be extended
to other nonlinear problems such as Navier-Stokes problems and phase field models.



\begin{thebibliography}{99}

\bibitem{Adams}
R. A. Adams: {\em Sobolev Spaces}. Academic Press, New York, 1975.

\bibitem{BankDupont}
R. E. Bank, T. Dupont: {\em An optimal order process for solving finite element equations}.
Math. Comp., 36 (1981), 35--51.

\bibitem{Bramble}
J. H. Bramble: {\em Multigrid Methods}. Pitman Research Notes in Mathematics, Vol. 294,
John Wiley and Sons, 1993.

\bibitem{BramblePasciak}
J. H. Bramble, J. E. Pasciak: {\em New convergence estimates for multigrid algorithms}. Math. Comp., 49 (1987), 311--329.

\bibitem{BrambleZhang}
J. H. Bramble, X. Zhang: {\em The Analysis of Multigrid Methods}. Handbook of Numerical Analysis,
2000, 173--415.

\bibitem{BrandtMcCormickRuge}
A. Brandt, S. McCormick, J. Ruge: {\em Multigrid methods for differential eigenproblems}.
SIAM J. Sci. Stat. Comput., 4(2) (1983), 244--260.

\bibitem{BrennerScott}
S. Brenner, L. Scott: {\em The Mathematical Theory of Finite Element Methods}. New York, Springer-Verlag, 1994.

\bibitem{Ciarlet}
P. G. Ciarlet: {\em The Finite Element Method for Elliptic Problem}.
North-holland Amsterdam, 1978.

\bibitem{Hackbusch_Book}
W. Hackbusch: {\em Multi-grid Methods and Applications}. Springer-Verlag, Berlin, 1985.

\bibitem{HuangShiTangXue}
Y. Huang, Z. Shi, T. Tang, W. Xue: {\em A multilevel successive
iteration method for nonlinear elliptic problem}.  Math. Comp., 73 (2004), 525--539.

\bibitem{JXXX}
S. Jia, H. Xie, M. Xie, F. Xu: {\em A Full Multigrid Method for Nonlinear Eigenvalue Problems}. Sci. China Math.,
59 (2016), 2037--2048.

\bibitem{LinXie}
Q. Lin, H. Xie: {\em A multi-level correction scheme for eigenvalue problems}.
Math. Comp., 84(291) (2015), 71--88.

\bibitem{LXX}
Q. Lin, H. Xie, F. Xu: {\em Multilevel correction adaptive finite element method for semilinear elliptic equation}.
Appl. Math., 60(5) (2015), 527--550.

\bibitem{ScottZhang}
L. Scott, S. Zhang: {\em Higher dimensional non-nested multigrid methods}. Math. Comp., 58 (1992), 457--466.

\bibitem{ToselliWidlund}
A. Toselli, O. Widlund: {\em Domain Decomposition Methods: Algorithm and Theory}.
Springer-Verlag, Berlin Heidelberg, 2005.

\bibitem{shaidurov1995multigrid}
V. V. Shaidurov: {\em Multigrid methods for finite elements}. Kluwer Academic Publics, Netherlands, 1995.

\bibitem{Xie_IMA}
H. Xie: {\em A type of multilevel method for the Steklov eigenvalue problem}.
IMA J. Numer. Anal., 34 (2014), 592--608.

\bibitem{Xie_JCP}
H. Xie: {\em A multigrid method for eigenvalue problem}. J. Comput. Phys., 274 (2014), 550--561.

\bibitem{Xie_Nonlinear}
H. Xie: {\em A multigrid method for nonlinear eigenvalue problems (in Chinese)}. Sci Sin Math, 45 (2015), 1193--1204.

\bibitem{XieXie}
H. Xie, M. Xie: {\em A Multigrid Method for the Ground State Solution of Bose-Einstein
Condensates}. Commun. Comput. Phys., 19(3) (2016), 648--662.

\bibitem{Xu}
J. Xu: {\em Iterative methods by space decomposition and subspace
correction}. SIAM Review, 34(4) (1992), 581--613.

\bibitem{J Xu1}
J. Xu: {\em Two-grid discretization techniques for linear and nonlinear PDEs}.
SIAM J. Numer. Anal., 33(5) (1996), 1759--1777.

\bibitem{J Xu2}
J. Xu: {\em A novel two-grid method for semilinear elliptic equations}.
SIAM J. Sci. Comput., 15(1) (1994), 231--237.

\end{thebibliography}
\end{document}